\documentclass{birkjour}
\usepackage{amsmath}
\usepackage{amssymb}
\usepackage{color}
\usepackage{tikz}
\usetikzlibrary{matrix}
\usepackage{array,multirow,graphicx}
\usepackage{pifont}

\newtheorem{thm}{Theorem}[section]
\newtheorem{cor}[thm]{Corollary}
\newtheorem{lem}[thm]{Lemma}
\newtheorem{prop}[thm]{Proposition}
\theoremstyle{definition}
\newtheorem{defn}[thm]{Definition}
\theoremstyle{remark}

\newtheorem*{ex}{Example}
\numberwithin{equation}{section}

\newcommand{\T}{\mathcal{T}}

\renewcommand{\O}{\mathcal{O}}

\newcommand{\R}{\mathbb{R}}

\title{Betweenness and Nonbetweenness}

\author[Atkins]{Ross Atkins}
\address{
University of Oxford \\
Department of Statistics \\
1 South Parks Road \\
Oxford OX1 3TG \\
United Kingdom}
\email{ross.atkins@univ.ox.ac.uk}

\newcommand{\x}{\mathbf{x}}
\newcommand{\y}{\mathbf{y}}
\newcommand{\z}{\mathbf{z}}
\newcommand{\ord}{\mbox{ord}}

\newcommand{\NBET}{\mbox{NBET}}
\newcommand{\BET}{\mbox{BET}}
\newcommand{\nbet}{\mbox{nbet}}
\newcommand{\bet}{\mbox{bet}}
\newcommand{\ETP}{\mbox{ETP}}
\newcommand{\EPT}{\mbox{EPT}}
\newcommand{\cat}{\mbox{cat}}

\begin{document}

\begin{abstract}
The betweenness function $\bet(n)$ is the minimum number of total orderings of $n$ objects such that for any three distinct objects $a$, $b$ and $c$, there is an ordering in which $b$ is between $a$ and $c$. The nonbetweenness function $\nbet(n)$ is the minimum number of total orderings such that for any three distinct objects $a$, $b$ and $c$, there is an ordering in which $b$ is not between $a$ and $c$. We show that $\nbet(n) = \left\lceil \log_2\log_2n \right\rceil+1$ and $\bet(n) = \Theta(\log n)$. Betweenness and Nonbetweenness are specific cases of a more general extreme value function called the \emph{extreme ternary constraint function} (definition given in Section~\ref{sec:etp}). The asymptotic value of this generalisation is computed using the values of $\nbet(n)$ and $\bet(n)$. This result demonstrates that the minimum size of a set of rooted phylogenetic trees is consistent with all phylogenetic triplets is $\Theta(\log\log n)$. 
\end{abstract}

\maketitle

\section{Notation and Definitions}

Our motivation for studying betweenness and nonbetweenness comes from the application to phylogenetic triplets \cite{Semple2003phylogenetics}. However, since betweenness and nonbetweenness are rich topics in their own right, the phylogenetic applications are left until Section~\ref{sec:phylo}. In this section we define betweenness and nonbetweenness and state the two main theorems, Theorems~\ref{thm:nbet} and~\ref{thm:bet}, which are proven in Section~\ref{sec:nbet} and Section~\ref{sec:bet} respectively. Section~\ref{sec:etp} focuses on a generalisation called $\ETP[\Pi]$ (the extreme ternary constraint problem).

\begin{defn}
Let $S_3 = \{ 123,132,213,231,312,321 \}$ be the group of permutations on $3$ objects.  
Let $n \geq 3$ be an integer and let $[n]$ denote the set of the first $n$ positive integers.
	$$ [n] = \{ 1,2,3, \ldots ,n \}. $$
An \emph{ordering} (total ordering on $[n]$) is a permutation $\phi : [n] \rightarrow [n]$. 
For an ordering $\phi$ and any distinct $x,y \in [n]$, we sometimes say \emph{$x$ comes before $y$ in $\phi$} to mean $\phi(x)<\phi(y)$. 
Otherwise $\phi(x)>\phi(y)$ and we would say \emph{$x$ comes after $y$}. 
We sometimes express $\phi$ as the $n$-tuple:
	$$ \phi = \big( \phi^{-1}(1) , \phi^{-1}(2) , \phi^{-1}(3) , \ldots , \phi^{-1}(n) \big). $$ 
$$ \mbox{For example, } \qquad \phi = (2,3,1) \quad \Longleftrightarrow \quad \phi(1) = 3 \; , \; \phi(2) = 1 \; \mbox{ and } \; \phi(3) = 2. $$ 
A \emph{ternary constraint} is a triple $\x = (x_1,x_2,x_3)$ of distinct integers: $x_1,x_2,x_3 \in [n]$. For an ordering $\phi$ and a ternary constraint $\x = (x_1,x_2,x_3)$, let the \emph{relative order} of $\x$ given by $\phi$ be 
	$$ \ord(\phi,\x) = abc $$ 
where $abc \in S_3$ such that $\phi(x_a) < \phi(x_b) < \phi(x_c)$.
\end{defn}

\begin{defn}
An ordering $\phi$ is said to \emph{between-satisfy} a ternary constraint $\x = (x_1,x_2,x_3)$ if and only if $x_2$ comes before one of $x_1$ or $x_3$, and comes after the other. \emph{i.e.} $\phi$ between-satisfies $\x$ if and only if 
	$$ \ord(\phi,\x) = 123 \quad \mbox{or} \quad \ord(\phi,\x) = 321. $$
An ordering is said to \emph{nonbetween-satisfy} a constraint if and only if it does not between-satisfy the constraint. \emph{i.e.} 
	$$ \phi \mbox{ nonbetween-satisfies } \x \quad \Longleftrightarrow \quad \ord(\phi,\x) \in \{ 132 , 213 , 231 , 312 \}. $$
Any set of orderings on $[n]$ is called an \emph{order-system}. An order-system, $\Phi$, is said to \emph{between-satisfy} a constraint, $\x$, if and only if $\x$ is between-satisfied by at least one $\phi \in \Phi$. Similarly, $\Phi$ \emph{nonbetween-satisfies} $\x$, if and only if $\x$ is nonbetween-satisfied by at least one $\phi \in \Phi$.
\end{defn}

\begin{ex}
For $n=5$ consider the order-system $\Phi = \{ \phi , \psi \}$ where $\phi = (1,2,3,4,5)$ and $\psi = (4,5,1,2,3)$. \emph{i.e.} $\phi(i)=i$ for all $i$ and: 
	$$ \psi(4)=1, \psi(5) = 2, \psi(1)=3, \psi(2)=4 \mbox{ and } \psi(3)=5. $$
Also consider the constraints $\x = (1,2,3)$, $\y = (2,4,5)$ and $\z = (1,5,4)$. We can verify that $\Phi$ between-satisfies $\x$, $\y$ and $\z$ because: 
\begin{itemize}
\item Both $\phi$  and $\psi$ between-satisfy $\x = (1,2,3)$ because $\ord(\phi,\x) = \ord(\psi,x) = 123$. 
\item $\phi$ between-satisfies $\y = (2,4,5)$ because $\phi(2) < \phi(4) < \phi(5)$ (i.e. $\ord(\phi,\y) = 123$). 
\item $\psi$ between-satisfies $\z = (1,5,4)$ because $\psi(1) > \psi(5) > \psi(4)$ (i.e. $\ord(\psi,\z) = 321$). 
\end{itemize} 
We can also say that $\Phi$ nonbetween-satisfies $\y$ and $\z$ because: 
\begin{itemize}
\item $\psi$ nonbetween-satisfies $\y$ since $4$ comes before both $2$ and $5$ in $\psi$ ($\ord(\psi,\y) = 231$). 
\item $\phi$ nonbetween-satisfies $\z$ since $5$ comes after both $1$ and $4$ in $\phi$ ($\ord(\phi,\z) = 132$). 
\end{itemize} 
However $\Phi$ does not nonbetween-satisfy $\x$ because neither $\phi$ nor $\psi$ nonbetween-satisfy $\x$. 
\end{ex}

\begin{defn}
For any integer $n \geq 3$:
\begin{itemize}
\item Let $\NBET = \NBET(n)$ be the collection of all order-systems on $[n]$ which nonbetween-satisfy all ternary constraints. 
\item Let $\BET = \BET(n)$ be the collection of all order-systems on $[n]$ which between-satisfy all ternary constraints.
\item Let $\nbet(n)$ and $\bet(n)$ denote the minimum size of any order-system in $\NBET(n)$ and $\BET(n)$ respectively. \emph{i.e.} 
	$$ \nbet(n) = \min \big\{ |\Phi| : \Phi \in \NBET(n) \big\} \qquad \mbox{and} \qquad 
	\bet(n) = \min \big\{ |\Phi| : \Phi \in \BET(n) \big\}. $$
\end{itemize}
\end{defn}

Figure~\ref{fig:tableOfValues} displays the values of $\nbet(n)$ and $\bet(n)$ for all $n$ up to $7$. Theorem~\ref{thm:nbet} presents the precise value of $\nbet(n)$x and Theorem~\ref{thm:bet} presents the asymptotics of $\bet(n)$ as $n \rightarrow \infty$.

\begin{figure}
$$ \begin{array}{l|cccccc}
n & 3 & 4 & 5 & 6 & 7 \\
\hline \\[-3mm] 
\bet(n) & 3 & 4 & 5 & 5 & 5 \\
\nbet(n) & 2 & 2 & 3 & 3 & 3 
\end{array} $$
\caption{The table of $\bet(n)$ and $\nbet(n)$ for $3 \leq n \leq 7$. The values of $\nbet$ come from Theorem~\ref{thm:nbet} and the values of $\bet$ are computed in Appendix~\ref{app:computing_bet(n)}.}
\label{fig:tableOfValues}
\end{figure}

\begin{thm}
\label{thm:nbet}
For all $n \geq 3$, 
	$$ \nbet(n) = \left\lceil \log_2\log_2n \right\rceil+1. $$
\end{thm}

\begin{thm}
\label{thm:bet}
As $n \rightarrow \infty$, 
	$$ \bet(n) = \Theta( \log n ). $$
\end{thm}

Theorem~\ref{thm:nbet} is proven in Section~\ref{sec:nbet} and Theorem~\ref{thm:bet} is proven in Section~\ref{sec:bet}. In Section~\ref{sec:etp}, the notion of betweenness and nonbetweenness is generalised, and the asymptotic value of the corresponding extreme value is computed. A phylogenetic application of these results is finally given in Section~\ref{sec:phylo}. We now conclude this section with a Proposition.

\begin{prop}
\label{prop:wlog}
If $\Phi \in \BET(n)$, then for any permutation $\sigma \in S_n$ we have $\Phi\sigma \in \BET(n)$, where 
	$$ \Phi\sigma = \{ \phi \circ \sigma : \phi \in \Phi \}. $$
Similarly, $\Phi \in \NBET(n)$ if and only if $\Phi\sigma \in \NBET(n)$.
\end{prop}
\begin{proof}
For an arbitrary constraint $\x = (x_1,x_2,x_3)$, consider the constraint $\x^\prime = (\sigma(x_1),\sigma(x_2),\sigma(x_3))$. For any ordering $\phi$, if $abc \in S_3$ is such that $\phi(\sigma(x_a)) < \phi(\sigma(x_b)) < \phi(\sigma(x_c))$ then 
	$$ \ord(\phi \circ \sigma,\x) = abc = \ord(\phi,\x^\prime). $$
\begin{itemize}
\item If $\Phi \in \BET(n)$ then there is some $\phi \in \Phi$  such that $\phi$ between-satisfies $\x^\prime$ and so $\ord(\phi,\x^\prime)$ is either $123$ or $321$. Therefore $\Phi\sigma \in \BET(n)$ because 
	$$ \ord(\phi \circ \sigma,\x) = \ord(\phi,\x^\prime) \in \{ 123,321 \}. $$
\item If $\Phi \in \NBET(n)$ then there is some $\phi \in \Phi$  such that $\phi$ nonbetween-satisfies $\x^\prime$ and so $\ord(\phi,\x^\prime)$ is $132$, $213$, $231$ or $312$. Therefore $\Phi\sigma \in \NBET(n)$ because 
	$$ \ord(\phi \circ \sigma,\x) = \ord(\phi,\x^\prime) \in \{ 132,213,231,312 \}. $$
\end{itemize}
\end{proof}

\section{Nonbetweenness}
\label{sec:nbet}

The results in this section are dedicated to the proof of Theorem~\ref{thm:nbet}. The proofs in this section fundamentally rely on Theorem~\ref{thm:ES} which is due to Erd\H{o}s and Szekeres \cite{Erdos1935combinatorial}. A proof is given here for completeness. Lemma~\ref{lem:ES_is_tight} demonstrates that the bound given in the Erd\H{o}s-Szekeres theorem is tight. 

\begin{thm} \cite{Erdos1935combinatorial}
Any sequence of $m^2+1$ real numbers contains a monotonic subsequence of length $m+1$. 
\label{thm:ES}
\end{thm}
\begin{proof}
Let $x_1,x_2, \ldots ,x_{m^2+1}$ be an arbitrary sequence of real numbers and let $f(i)$ denote the maximum length of all non-decreasing subsequences beginning with the $i^{\mbox{\scriptsize th}}$ number $x_i$. For $i<j$, if $f(i)=f(j)$ then $x_i>x_j$. Therefore for any $k$, the subsequence $(x_i : f(i)=k)$ must be decreasing. Either $f(i)>m$ for some $i$, or $f(i) \leq m$ for all $i$. In the former case, there is a non-decreasing subsequence of length $m+1$ beginning with $x_i$. In the later case $f: [m^2+1] \rightarrow [m]$, so by the pigeon-hole principle there must be some $k \in [m]$ such that $|\{ x_i : f(x_i)=k \}| >m$. In this case, the subsequence $(x_i : f(i)=k)$ is a decreasing subsequence of length at least $m+1$. 
\end{proof}

\begin{defn}
A sequence of points in $\R^d$ is called monotonic if and only if it is monotonic in each coordinate.
\end{defn}

\begin{cor} 
Let $d$ and $m$ be positive integers. Any sequence of $N=m^{2^d}+1$ points in $\R^d$ contains a monotonic subsequence of length $m+1$. 
\label{cor:ES_high_dim}
\end{cor}
\begin{proof} (by Induction on $d$) The base case, $d=1$, is exactly the Erd\H{o}s-Szekeres theorem (Theorem~\ref{thm:ES}). For the inductive step, $d \geq 2$, we can use Theorem~\ref{thm:ES} to find a subsequence of length $M$ which is monotonic in the first coordinate, such that  
	$$ M = \sqrt{N-1}+1 = m^{2^{d-1}}+1. $$ 
Now we can apply the inductive assumption to this sequence (the first coordinates can now be guaranteed to be monotonic, so they are ignored) to find the required subsequence of length $m+1$. 
\qedhere
\end{proof}

\begin{lem}
\label{lem:ES_is_tight}
For positive integers $d,m$ there exists a sequence $x_0,x_1 \ldots ,x_{N-1}$, of $N=m^{2^d}$ points in $\R^d$ with no monotonic subsequence of length $m+1$. Moreover, one can construct such a sequence for which the $k^{\mbox{\scriptsize th}}$ coordinates are a rearrangement of $[N]$  for each $k=1,2, \ldots ,d$.
\end{lem}
In the following proof, we construct a suitable sequence in which the $k^{\mbox{\scriptsize th}}$ coordinates are a rearrangement of $\{ 0,1,2, \ldots , N-1 \}$, not the conventional $[N] = \{ 1,2,3, \ldots ,N \}$. This can be overcome by simply adding one to each coordinate. 
\begin{proof} When $m=1$ the result is trivial, so we only consider when $m \geq 2$. We construct a suitable sequence by induction on $d$. 
\begin{description}
\item[Base case] For $d=1$ we have $N=m^2$. For all $a,b$ in the range $\{ 0,1,2, \ldots ,m-1 \}$, let 
	$$ x_{am+b} = am+(m-b-1). $$ 
Now let $\left( x_{k_i} \right)_{i=1}^{m+1}$ be an arbitrary subsequence of $(x_k)_{k=0}^{N-1}$ with length $m+1$. 
\begin{itemize}
\item By the pigeon-hole principle, there must be two terms $x_{k_i},x_{k_{i^\prime}}$ (wlog $k_i<k_{i^\prime}$) such that $k_i$ and $k_{i^\prime}$ have equal remainders upon division by $m$. We have $x_{k_i}-x_{k_{i^\prime}} = k_i-k_{i^\prime}<0$ and so the subsequence cannot be decreasing.
\item Similarly by the pigeon-hole principle, there are two terms $x_{k_j},x_{k_{j^\prime}}$ (wlog $k_j<k_{j^\prime}$) such that $\left\lfloor k_j/m \right\rfloor = \left\lfloor k_{j^\prime}/m \right\rfloor$. We have $x_{k_j}-x_{k_{j^\prime}} = k_{j^\prime}-k_j>0$ and so the subsequence cannot be increasing.
\end{itemize}
Therefore all monotonic subsequences have length at most $m$.
\item[Inductive step:] Let $M = \sqrt{N} = m^{2^{d-1}}$. By the inductive assumption, let $(y_i)_{i=0}^{M-1}$ be a sequence of points in $\R^{d-1}$ with no monotonic subsequence of length $m+1$ such that the $k^{\mbox{\scriptsize th}}$ coordinates are $\{ 0,1, \ldots ,M-1 \}$ in some order for each $k=1, \ldots ,d-1$. Let $y_i^k$ be the $k^{\mbox{\scriptsize th}}$ coordinate of $y_i$. We construct a sequence $(x_i)_{i=0}^{N-1}$ where the $k^{\mbox{\scriptsize th}}$ coordinate of $x_i$ is denoted $x_i^k$. For all $a,b \in \{ 0,1,2, \ldots ,M-1 \}$ define  
	$$ x_{aM+b}^k = y_a^kM + y_b^k, $$
for $k=1,2, \ldots ,d-1$. Also define $x_{aM+b}^d = aM + (M-b-1)$. Now partition the sequence into $M$ equal parts (each part has size $M$) in the following way. 
	$$ P_j = \left( x_{jM+i} \right)_{i=0}^{M-1} $$
for each $j=0,1,2, \ldots ,M-1$.  Now we show by contradiction that the sequence $(x_i)_{i=0}^{N-1}$ contains no monotonic subsequence of length $m+1$:
\begin{itemize} 
\item 
Any subsequence which is increasing in its $d^{\mbox{\scriptsize th}}$ coordinate must have at most one term in each $P_j$. However, this could not have length $m+1$ and be monotonic in the other coordinates by the inductive assumption. 
\item 
Any subsequence which is decreasing in its $d^{\mbox{\scriptsize th}}$ coordinate must be contained within a single $P_j$. Such a subsequence could not have length $m+1$ and be monotonic in the other coordinates by the inductive assumption. 
\end{itemize}
\end{description}
\end{proof}

We now have all the tools required to prove Lemmas~\ref{lem:nbet_lower_bound} and~\ref{lem:nbet_upper_bound}. Together, these lemma imply Theorem~\ref{thm:nbet}. 

\begin{lem}
\label{lem:nbet_lower_bound}
$\nbet(n) \geq \log_2\log_2n+1$.
\end{lem}
\begin{proof}
Let $\Phi = \{ \phi_0,\phi_1,\phi_2, \ldots ,\phi_d \} \in \NBET(n)$ be an order-system of size $\nbet(n) = d+1$. Without loss of generality, let $\phi_0$ be the identity ordering (if not then we can let $\sigma = \phi_0^{-1}$ in Proposition~\ref{prop:wlog} and consider $\Phi\sigma \in \NBET(n)$). Now consider the sequence $(\mathbf{p}_i)_{i=1}^n$ of points in $\R^d$ defined by 
	$$ \mathbf{p}_{i} = (\phi_1(i),\phi_2(i), \ldots ,\phi_d(i)). $$
If this sequence contained a monotonic subsequence of length $3$, say $(\mathbf{p}_a,\mathbf{p}_b,\mathbf{p}_c)$, then $\phi_j$ would between-satisfy the constraint $(a,b,c)$ for all $j$, but this would contradict $\Phi \in \NBET(n)$. Therefore the sequence $(\mathbf{p}_i)_{i=1}^n$ contains no monotonic subsequence of length $3$. By Corollary~\ref{cor:ES_high_dim} (with $m=2$) this means $n \leq 2^{2^{d}}$. Hence
	$$ \nbet(n) = d+1 \geq \log_2\log_2n+1. $$ 
\end{proof}

\begin{lem}
\label{lem:nbet_upper_bound}
$\nbet \left( 2^{2^d} \right) = d+1$.
\end{lem}
\begin{proof}
Let $n=2^{2^n}$. Because of Lemma~\ref{lem:nbet_lower_bound}, it suffices to construct an order-system $\Phi \in \NBET(n)$ with $|\Phi|=d+1$. Using Lemma~\ref{lem:ES_is_tight} (with $m=2$), let $(\mathbf{p}_i)_{i=1}^n$ be a sequence of $n$ points in $\R^d$ containing no monotonic subsequence of length $3$, such that for each coordinate $k$, the $k^{\mbox{\scriptsize th}}$ coordinates are an ordering of $[n]$. Let $\phi_0$ be the identity ordering and for each $k=1,2, \ldots ,d$, let $\phi_k$ be the ordering given by the $k^{\mbox{\scriptsize th}}$ coordinate of this sequence. It suffices to show 
	$$ \Phi = \{ \phi_i : i=0,1,2, \ldots ,d\} \in \NBET(n). $$
Let $\x = (x_1,x_2,x_3)$ be an arbitrary constraint. If $x_2$ was not between $x_1$ and $x_3$ then $\phi_0$ would nonbetween-satisfy $\x$. So we assume $x_2$ is between $x_1$ and $x_3$, and without loss of generality let $x_1<x_2<x_3$. Since $\big(\mathbf{p}(x_1),\mathbf{p}(x_2),\mathbf{p}(x_3)\big)$ is not monotonic in $\R^d$; there must be some coordinate $k$ such that $\phi_k(x_2)$ is not between $\phi_k(x_1)$ and $\phi_k(x_3)$, and so $\phi_k$ nonbetween-satisfies $\x$. 
\end{proof}

\begin{proof}[Proof of Theorem~\ref{thm:nbet}] The case that $n$ is a double power of $2$ is exactly Lemma~\ref{lem:nbet_upper_bound}. In general, let $d$ be the positive integer such that $2^{2^{d-1}} < n \leq 2^{2^d}$ (i.e. $d = \lceil \log_2 \log_2n \rceil$). So
  \begin{align*}
  d & < \log_2\log_2n+1 \\
  & \leq \nbet(n) 
    & \mbox{(Lemma~\ref{lem:nbet_lower_bound})} \\ 
  & \leq \nbet \left( 2^{2^d} \right) 
    & \mbox{($\nbet$ is non-decreasing)} \\ 
  & = d+1.
    & \mbox{(Lemma~\ref{lem:nbet_upper_bound})} 
  \end{align*}
The value $\nbet(n)$ is an integer strictly greater than $d$ and at most $d+1$. Therefore
	$$ \nbet(n)=d+1=\lceil \log_2 \log_2n \rceil +1. $$
\end{proof}

\section{Betweeness}
\label{sec:bet}

In this section, we prove Theorem~\ref{thm:bet}. The lower bound is established in Lemma~\ref{lem:bet_lower_bound} and the upper bound is established in Lemma~\ref{lem:bet_upper_bound}.

\begin{lem}
\label{lem:bet_lower_bound}
$\bet(n) \geq \log_2(n-1)+1$.
\end{lem}
\begin{proof}
Let $\bet(n) = k+1$, let $\Phi = \{ \phi_0,\phi_1,\phi_2, \ldots ,\phi_k \} \in \BET(n)$ and without loss of generality, let $\phi_0$ be the identity ordering (if not then by Proposition~\ref{prop:wlog} we can let $\sigma = \phi_0^{-1}$ and consider $\Phi\sigma \in \BET(n)$). Note that $\phi_0$ does not between-satisfy $(x,n,y)$ for any $x,y$. Consider the function $\mathbf{f} : [n-1] \rightarrow \{ 0,1 \}^k$ defined by $\mathbf{f}(x) = (I_1,I_2, \ldots ,I_k)$ where $I_i$ is the indicator function of $\phi_i(x)<\phi_i(n)$, for all $1 \leq i \leq k$. 
	$$ I_i = \left\{ \begin{array}{cl}
	1 & : \mbox{ if } \phi_i(x)<\phi_i(n) \\
	0 & : \mbox{ otherwise.} 
	\end{array} \right. $$
For any distinct $x,y \in [n-1]$ we must have $\mathbf{f}(x) \not= \mathbf{f}(y)$, otherwise the constraint $(x,n,y)$ would not be between-satisfied by $\Phi$. Therefore $\mathbf{f}$ is injective and its domain is smaller than or equal to its codomain. So
	$$ n-1 \leq \left| \{ 0,1 \}^k \right| = 2^{\bet(n)-1}. $$
\end{proof}

\begin{lem}
\label{lem:bet_upper_bound}
$\bet(n) \leq 2 \left\lceil \log_2(n) \right\rceil$.
\end{lem}
\begin{proof}
Since $\bet(n)$ is a non-decreasing integer function, it suffices to consider only the case that $n = 2^k$ for some integer $k \geq 2$. We do this by explicitly constructing an order-system $\Phi \in \BET(2^k)$ with $|\Phi|=2k$. First let $\{ \psi_i \}_{i=1}^k$ be a set of $k$ distinct orderings such that for each $x,y \in [n]$, there is at least one $\psi_i$ such that $\psi_i(x) \leq \frac{n}{2} < \psi_i(y)$ or $\psi_i(y) \leq \frac{n}{2} < \psi_i(x)$.\footnote{
There are many ways to find such set $\{ \psi_i \}_{i=1}^k$. One way is to write the elements of $[n]$ in base two, and then let $\psi_i$ be the $i^{\mbox{\scriptsize th}}$ cyclic shift of the digits.
}
For each $\psi_i$, let $\phi_i$ be an ordering defined by 
	$$ \phi_i(x) = \left\{ \begin{array}{cl}
	\psi_i(x) - n/2 & : \mbox{ if } \psi_i(x)>\frac{n}{2} \\
	\psi_i(x) + n/2 & : \mbox{ otherwise.} \end{array} \right. $$
Now consider an arbitrary ternary constraint $(x,\alpha,y)$, and let $i$ be such that $\psi_i(x) \leq \frac{n}{2} < \psi_i(y)$ or $\psi_i(y) \leq \frac{n}{2} < \psi_i(x)$. If $\psi_i$ does not between-satisfy $(x,\alpha,y)$ then there are four cases:
\begin{itemize}
\item If $\psi(\alpha) < \psi_i(x) \leq \frac{n}{2} < \psi_i(y)$ then $\phi_i(y) \leq n/2 < \phi_i(\alpha) < \phi_i(x)$.
\item If $\psi_i(x) \leq \frac{n}{2} < \psi_i(y) < \psi(\alpha)$ then $\phi_i(y) < \psi(\alpha) \leq n/2 < \phi_i(x)$.
\item If $\psi(\alpha) < \psi_i(y) \leq \frac{n}{2} < \psi_i(x)$ then $\phi_i(x) \leq n/2 < \phi_i(\alpha) < \phi_i(y)$.
\item If $\psi_i(y) \leq \frac{n}{2} < \psi_i(x) < \psi(\alpha)$ then $\phi_i(x) < \psi(\alpha) \leq n/2 < \phi_i(y)$.
\end{itemize}
In any case, if $\psi_i$ does not between-satisfy $(x,\alpha,y)$ then $\phi_i$ does. Therefore 
	$$ \Phi = \bigcup_{i=1}^k \big\{ \phi_i , \psi_i \big\} \in \BET(n). $$
\end{proof}

By considering the first few values $n$ (Figure~\ref{fig:tableOfValues}), it seems that $\bet(n)$ is close to $2\log_2n$. The exact value of $\bet(n)$ for $n \geq 8$ is left as an open question. 

\section{The Extreme Ternary Constraint Problem}
\label{sec:etp}

The extreme ternary constraint problem is a generalisation of both betweeness and nonbetweenness. In this section we define $\ETP[\Pi]$ (Definition~\ref{defn:etp}) and we determine the asymptotics of $p_{\Pi}(n)$ for all proper subsets $\Pi \subset S_3$ (Theorem~\ref{thm:etp}). 

\begin{defn}
\label{defn:etp}
Let $\Pi \subset S_3$ be a non-empty set of permutations.
\begin{itemize}
\item An order-system $\Phi$ on $n$ is said to $\Pi$-solve the extreme ternary constraint problem if for every constraint $\x$, there exists an order $\phi \in \Phi$ such that $\ord(\phi,\x) \in \Pi$.
\item Let $\ETP[\Pi]$ denote the set of all order-systems that $\Pi$-solve the extreme ternary constraint problem.
\item Let $p_{\Pi}(n)$ denote the minimum size of an order-system in $\ETP[\Pi]$.
\end{itemize}
\end{defn}

We only consider non-empty subsets $\Pi$ because $p_\emptyset(n)$ is not defined.

\begin{prop}
\label{prop:etp_wlog}
For $\Pi \subseteq S_3$ and any permutation $\sigma \in S_3$ we have $\ETP[\Pi] = \ETP[\sigma\Pi]$ where $\sigma\Pi = \{ \sigma \circ \pi \: | \: \pi \in \Pi \}$, and so $p_{\Pi}(n) = p_{\sigma\Pi}(n)$.
\end{prop}
\begin{proof}
Since $\Pi = \sigma^{-1}(\sigma\Pi)$, it suffices to show that $\ETP[\Pi] \subseteq \ETP[\sigma\Pi]$. To show this, we will prove that for any $\Phi \in \ETP[\Pi]$, we must have $\Phi \in \ETP[\sigma\Pi]$. For any constraint $\x = (x_1,x_2,x_3)$ there exists a constraint $\x^\prime = (x_{\sigma^{-1}(1)},x_{\sigma^{-1}(2)},x_{\sigma^{-1}(3)})$. By construction, for any ordering $\phi$ we have
	$$ \ord(\phi,\x) = \sigma \circ \ord(\phi,\x^\prime). $$ 
Therefore if $\Phi \in \ETP[\Pi]$ then for any constraint $\x$, there exists some $\phi \in \Phi$ such that $\ord(\phi,\x^\prime) \in \Pi$. For this ordering $\phi$, we must have 
	$$ \ord(\phi,\x) = \sigma \circ \ord(\phi,\x^\prime) \in \sigma\Pi. $$
\end{proof}

\begin{defn}
For $a=1,2,3$ let $M_a$ denote the set of elements of $S_3$ with $a$ in the middle. i.e.
	$$ \begin{array}{c} 
	M_1 = \{ 213, 312 \} \\
	M_2 = \{ 123, 321 \} \\
	M_3 = \{ 132, 231 \} 
	\end{array} $$
\end{defn}

By definition: $\BET = \ETP[M_2]$ and $\NBET = \ETP[M_1 \cup M_3]$. So the extreme ternary constraint problem is a generalisation of betweenness and nonbetweenness. The asymptotics of $p_{(M_1 \cup M_3)}(n)$ and $p_{M_2}(n)$ are $\Theta(\log\log n)$ and $\Theta(\log n)$ respectively by Theorems \ref{thm:nbet} and \ref{thm:bet}. For $\Pi = S_3$, we trivially have $p_\Pi(n) = 1$ for all $n$, and for $\Pi = \emptyset$, the function $p_\Pi$ is not defined. For other sets $\Pi$, the asymptotics of $p_{\Pi}(n)$ are presented in Theorem~\ref{thm:etp}.

\begin{thm}
\label{thm:etp}
Let $\emptyset \not= \Pi \subset S_3$ and let $c$ be the number of distinct $M_a$ such that $\Pi \cap M_a = \emptyset$. 
\begin{itemize}
\item If $c=0$ then $p_{\Pi}(n) = 2$.
\item If $c=1$ then $p_{\Pi}(n) = \Theta(\log\log n)$.
\item If $c=2$ then $p_{\Pi}(n) = \Theta(\log n)$.
\end{itemize}
\end{thm}

The three parts of Theorem~\ref{thm:etp} are proven separately in Lemmas~\ref{lem:c=0}. \ref{lem:c=1} and \ref{lem:c=2}. 

\begin{lem}
\label{lem:c=0}
If $\Pi \not= S_3$ and $\Pi$ intersects each of $M_1$, $M_2$ and $M_3$, then $p_{\Pi}(n) = 2$.
\end{lem}
\begin{proof}
Since $S_3 \backslash \Pi \not= \emptyset$, for each ordering $\phi$ there will be some constraint $\x$ such that $\ord(\phi,\x) \in S_3 \backslash \Pi$. Therefore $\{ \phi \} \not\in \ETP[\Pi]$ for any ordering $\phi$. So $p_\Pi(n) \not= 1$ and thus $p_\Pi(n) \geq 2$. However for any order, $\phi$, suppose $\ord(\phi,\x) = \sigma \notin \Pi$ for some constraint $\x$. There must be some $a \in \{ 1,2,3 \}$ such that $\sigma \in M_a$. So let $M_a = \{ \sigma , \tau \}$ and let $\phi^R$ be the reverse of $\phi$. Since $\Pi \cap M_a \not= \emptyset$ we must have $\tau \in \Pi$ and (since $\tau$ is the reverse of $\sigma$) 
	$$ \ord(\phi^R,\x) = \tau \in \Pi. $$ 
Thus for any ordering $\phi$ and any constraint $\x$, either $\ord(\phi,\x) \in \Pi$ or $\ord(\phi^R,\x) \in \Pi$. For any $\phi$, we have $\{ \phi,\phi^R \} \in \ETP[\Pi]$. Hence $p_{\Pi}(n) = 2$.
\end{proof}

\begin{lem}
\label{lem:c=1}
If $\Pi$ intersects exactly $2$ of $\{ M_1, M_2, M_3 \}$, then for all $n \geq 3$
	$$ \left\lceil \log_2\log_2 n \right\rceil + 1 \leq p_{\Pi}(n) 
	\leq 2\left\lceil \log_2\log_2 n \right\rceil + 2. $$ 
\end{lem}
\begin{proof}
Let $a$ be the index such that $\Pi \cap M_a = \emptyset$. If $a \not= 2$ then we can apply Proposition~\ref{prop:etp_wlog} (with either $\sigma = 231$ or $\sigma = 312$) so that $\sigma\Pi$ intersects $M_1$ and $M_3$ but not $M_2$. So without loss of generality let us assume $\Pi \subseteq M_1 \cup M_3$. This implies $\ETP[\Pi] \subseteq \ETP[M_1 \cup M_3] = \NBET$ and therefore 
	$$ p_{\Pi}(n) \geq \nbet(n) = \left\lceil \log_2\log_2 n \right\rceil + 1. $$
To show the upper bound, consider some $\Phi \in \NBET(n)$ with $|\Phi| = \nbet(n) = \left\lceil \log_2\log_2n \right\rceil +1$. Let $\x = (x_1,x_2,x_3)$ be an arbitrary constraint. There must be some $\phi \in \Phi$ such that $x_2$ is not between $x_1$ and $x_3$ in $\phi$. Let $\ord(\phi,\x) = \sigma \in M_1 \cup M_3$.
\begin{itemize}
\item If $\sigma \in \Pi$ then $\phi$ $\Pi$-satisfies $\x$. 
\item If for $b =1$ or $3$, we have $\sigma \in M_b \backslash \Pi$ then $M_b = \{ \sigma , \tau \}$ where $\tau \in \Pi$ and $\tau$ is the reverse of $\sigma$. In this case $\ord(\phi^R,\x) = \tau \in \Pi$, where $\phi^R$ is the reverse of $\phi$.
\end{itemize}
So either $\ord(\phi,\x) \in \Pi$ or $\ord(\phi^R,\x) \in \Pi$. Hence $\Psi = \{ \phi , \phi^R | \phi \in \Phi \} \in \ETP[\Pi]$. Therefore
	$$ p_{\Pi}(n) \leq |\Psi| = 2\left\lceil \log_2\log_2 n \right\rceil + 2. $$
\end{proof}

\begin{lem}
\label{lem:c=2}
If $\Pi$ is non-empty and intersects exactly $1$ of $\{ M_1, M_2, M_3 \}$ then for all $n \geq 3$
	$$ \log_2 (n-1)+1 \leq p_{\Pi}(n) \leq 4 \left\lceil \log_2 n \right\rceil.$$ 
\end{lem}
\begin{proof}
Let $a$ be the index such that $\Pi \subseteq M_a$. If $a \not= 2$ then we can apply Proposition~\ref{prop:etp_wlog} (with either $\sigma = 231$ or $\sigma = 312$) so that $\sigma\Pi \subseteq M_2$. So without loss of generality let us assume $\Pi \cap M_1$ and $\Pi \cap M_3$ are empty, and so $\Pi \subseteq M_2$. Therefore $\ETP[\Pi] \subseteq \ETP[M_2] = \BET$ and therefore 
	$$ p_{\Pi}(n) \geq \bet(n) \geq \log_2(n-1)+1. $$ 
To show the upper bound, let $\Phi \in \BET(n)$ be chosen arbitrarily with $|\Phi| = \bet(n) \leq 2\left\lceil \log_2n \right\rceil$ (Lemma~\ref{lem:bet_upper_bound}). Let $\x = (x_1,x_2,x_3)$ be an arbitrary constraint. Since $\Phi$ solves $\BET$, there must be some $\phi \in \Phi$ such that $x_2$ is between $x_1$ and $x_3$ in $\phi$. Without loss of generality $M_2 = \{ \sigma , \tau \}$ and $\ord(\phi,\x) = \sigma$. 
If $\sigma \not\in \Pi$ then $\Pi = \{ \tau \}$ and $\ord(\phi^R,\x) = \tau \in \Pi$ where $\phi^R$ is the reverse of $\phi$. Either way $\ord(\phi,\x)$ or $\ord(\phi^R,\x)$ is in $\Pi$. Hence
	$$ \Psi = \{ \phi , \phi^R | \phi \in \Phi \} \in \ETP[\Pi] $$
and therefore $p_{\Pi}(n) \leq |\Psi| \leq 4 \left\lceil \log_2 n \right\rceil$.
\end{proof}

\section{An application to Phylogenetics}
\label{sec:phylo}

A phylogenetic tree is a rooted binary tree with leaves labelled by $[n]$. A phylogenetic triplet \cite{Semple2003phylogenetics} denoted $(a|b,c)$ is any triple of three distinct integers $a,b,c \in [n]$. A phylogenetic tree is said to be \emph{consistent} with the phylogenetic triplet, $(a|b,c)$, if the path from the leaf labelled $a$ to the root does not intersect the path between the leaves labelled $b$ and $c$. Phylogeneticists sometimes search for a tree or a network which is consistent with a given set of triplets \cite{Boekhout2009constructing,Critchlow1996triples,Habib2009level}. Certain sets of phylogenetic trees can be represented as a network with hybridization vertices \cite{Gusfield2014recombinatorics}. We consider the following question: what is the minimal size of a set of phylogenetic trees such that any phylogenetic triplet $(a|bc)$ is consistent with at least one tree in the set? This will be denoted $p(n)$ (see Definition~\ref{defn:ept}) and the asymptotic value is computed in Theorem~\ref{thm:ept}.

\begin{defn}
\label{defn:ept}
Let $\T$ be a set of rooted phylogenetic trees with $n$ leaves labelled by $[n]$. We say $\T \in \EPT$ if and only if for every triplet $(a|bc)$, at least one tree $T \in \T$ is consistent with $(a|bc)$. Let $p(n)$ be the minimal size of all sets of phylogenetic trees in $\EPT = \EPT(n)$. 
\end{defn}

\begin{defn}
For any ordering $\phi$ on $[n]$, let $\cat(\phi)$ be the rooted caterpillar with leaves labelled by $[n]$ in the order given by $\phi$ (i.e. the leaf which is the $i^{\mbox{\scriptsize th}}$ closest to the root is labelled $\phi^{-1}(i)$). For example if $n=4$ and $\phi = (3,1,2,4)$ then 
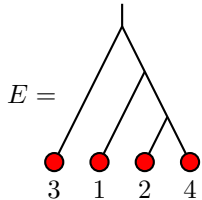
\begin{center}
\begin{tikzpicture}[scale=0.3]
	\draw[thick] (0,0) -- (3,6) -- (3,7);
	\draw[thick] (6,0) -- (3,6);
	\draw[thick] (2,0) -- (4,4);
	\draw[thick] (4,0) -- (5,2);
	\draw[thick,draw=black,fill=red] (0,0) circle (0.4) ;
	\draw[thick,draw=black,fill=red] (2,0) circle (0.4) ;
	\draw[thick,draw=black,fill=red] (4,0) circle (0.4) ;
	\draw[thick,draw=black,fill=red] (6,0) circle (0.4) ;
	\node[below] at (0,-0.4) {$3$};
	\node[below] at (2,-0.4) {$1$};
	\node[below] at (4,-0.4) {$2$};
	\node[below] at (6,-0.4) {$4$};
	\node at (-1,3) {$E = $};
\end{tikzpicture}
\end{center}
\end{defn}

\begin{prop}
\label{prop:ordering_to_cat}
Let $\Pi = \{ 123,132 \} \subseteq S_3$, let $\phi$ be any ordering and let $\x = (x_1,x_2,x_3)$ be any ternary constraint. We have $\ord(\phi,\x) \in \Pi$ if and only if $\cat(\phi)$ is consistent with the triplet $(x_1|x_2,x_3)$.
\end{prop}
\begin{proof}
We have $\ord(\phi,\x) \in \Pi$ if and only if $x_1$ comes before $x_2$ and $x_3$ in $\phi$. Equivalently $x_1$ is a leaf in $\cat(\phi)$ nearer to the root than $x_2$ and $x_3$. This is exactly what it means for a caterpillar to be consistent with $(x_1|x_2,x_3)$.
\end{proof}
\begin{prop}
\label{prop:p(n)<pPi(n)}
Let $\Pi = \{ 123,132 \}$. For any $\Phi \in \ETP[\Pi]$, the set of trees $\cat(\Phi) = \{ \cat(\phi) : \phi \in \Phi \}$ is in $\EPT$ and so
	$$ p(n) \leq p_{\Pi}(n). $$
\end{prop}

Proposition~\ref{prop:p(n)<pPi(n)} is a direct consequence of Proposition~\ref{prop:ordering_to_cat}. These two Propositions highlight a connection between the triplet problem and the ternary-constraint problem. The value of $p(n)$ would be exactly $p_{\{ 123,132 \} }(n)$ if we only allowed phylogenetic trees that were caterpillars. This relationship between the ternary constraint problem and the phylogenetic triplet problem is a known result \cite{VanIersel2014satisfying}.

\begin{defn}
For any rooted planar embedding $E$ of a phylogenetic tree with leaf label set $[n]$, let $t(E)$ be the total ordering on $[n]$ given by listing the leaf labels from left to right (with the root at the top). For example if 
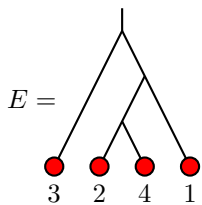
\begin{center}
\begin{tikzpicture}[scale=0.3]
	\draw[thick] (0,0) -- (3,6) -- (3,7);
	\draw[thick] (6,0) -- (3,6);
	\draw[thick] (2,0) -- (4,4);
	\draw[thick] (4,0) -- (3,2);
	\draw[thick,draw=black,fill=red] (0,0) circle (0.4) ;
	\draw[thick,draw=black,fill=red] (2,0) circle (0.4) ;
	\draw[thick,draw=black,fill=red] (4,0) circle (0.4) ;
	\draw[thick,draw=black,fill=red] (6,0) circle (0.4) ;
	\node[below] at (0,-0.4) {$3$};
	\node[below] at (2,-0.4) {$2$};
	\node[below] at (4,-0.4) {$4$};
	\node[below] at (6,-0.4) {$1$};
	\node at (-1,3) {$E = $};
\end{tikzpicture}
\end{center}
then $t(E) = (3,2,4,1)$.
\end{defn}

There are many\footnote{For any rooted binary tree with $n$ leaves, there are exactly $2^{n-1}$ different planar embeddings.} different planar embeddings of a phylogenetic tree. As $E$ varies over all planar embeddings of a fixed tree, $t(E)$ will result in a variety of different orders. For our purposes, it will not matter which planar embedding is used.

\begin{lem}
\label{lem:p>nbet} 
$p(n) \geq nbet(n)$.
\end{lem}
\begin{proof}
Let $E$ be any planar embedding of a phylogenetic tree $T$. If $T$ is consistent with $(x_2|x_1,x_3)$ then $x_1$ and $x_3$ must be on the same side of $x_2$ in any planar embedding of $T$. So in $t(E)$, $x_2$ is not between $x_1$ and $x_3$ and so $t(E)$ would nonbetween-satisfy $(x_1,x_2,x_3)$. \\
Now let $p=p(n)$, let $\{ T_1,T_2, \ldots T_p \} \in \EPT(n)$ and for each $i$ let $E_i$ be a planar embedding of $T_i$. Consider the order-system $\Phi = \{ t(E_1) , t(E_2), \ldots ,t(E_p) \}$. For any $\x = (x_1,x_2,x_3)$, there is some $T_i$ which is consistent with $(x_2|x_1,x_3)$ and so there is some $\phi = t(E_i) \in \Phi$ which nonbetween-satisfies $\x$. Hence $\Phi \in \NBET(n)$ and so $\nbet(n) \leq p$.
\end{proof}

\begin{thm}
\label{thm:ept}
$p(n) = \Theta(\log\log n)$.
\end{thm}
\begin{proof}
Let $\Pi = \{ 123,132 \}$. Proposition~\ref{prop:p(n)<pPi(n)} and Lemma~\ref{lem:p>nbet} provide the bounds:
	$$ \nbet(n) \leq p(n) \leq p_{\Pi}(n). $$
By Theorem~\ref{thm:nbet} we know $\nbet(n) = \lceil \log_2\log_2n \rceil + 1 = \Theta(\log\log n)$, and by Theorem~\ref{thm:etp} we know $p_{\Pi}(n) = \Theta(\log\log n)$ too. 
\end{proof}
The explicit lower and upper bounds obtained this way are 
	$$ \left\lceil \log_2\log_2n \right\rceil + 1 \leq p(n) \leq 2 \left\lceil \log_2\log_2n \right\rceil + 2. $$
It is certainly true that $p(n) \rightarrow \infty$ as $n \rightarrow \infty$, however even if we substitute an unrealistically large value of $n$, the result is surprisingly small. If $n$ is the number of atoms on Earth, then $p(n) \leq 18$. 

\bibliography{BNB}

\begin{thebibliography}{1}

\bibitem{Boekhout2009constructing}
Teun Boekhout, Ferry Hagen, Judith Keijsper, Steven Kelk, Leen Stougie, and
  Leo~Van Iersel.
\newblock Constructing level-2 phylogenetic networks from triplets.
\newblock {\em IEEE/ACM Trans Comput Biol Bioinform.}, 6:667--681, 2009.

\bibitem{Critchlow1996triples}
Douglas~E. Critchlow, Dennis~K. Pearl, and Chunlin Qian.
\newblock The triples distance for rooted bifurcating phylogenetic trees.
\newblock {\em Syst. Biol.}, 45:323--334, 1996.

\bibitem{Erdos1935combinatorial}
Paul Erd\H{o}s and George Szekeres.
\newblock A combinatorial problem in geometry.
\newblock {\em Composition Mathematica 2}, pages 463--470, 1935.

\bibitem{Gusfield2014recombinatorics}
Dan Gusfield.
\newblock {\em Recombinatorics}.
\newblock The MIT Press, Cambridge, Massachusetts, 2014.

\bibitem{Habib2009level}
Michel Habib and Thu-Hien To.
\newblock {\em Level-$k$ phylogenetic networks are constructible from a dense
  triplet set in polynomial time}, volume 5577, pages 275--288.
\newblock Springer Berlin Heidelberg, 2009.

\bibitem{VanIersel2014satisfying}
Leo~Van Iersel, Steven Kelk, Nela Lekic, and Simone Linz.
\newblock Satisfying ternary permutation constraints by multiple linear orders
  or phylogenetic trees.
\newblock {\em arXiv:1503.00368}, 2014.

\bibitem{Semple2003phylogenetics}
Charles Semple and Mike Steel.
\newblock {\em Phylogenetics}.
\newblock Oxford University Press, 2003.

\end{thebibliography}
\bibliographystyle{plain}

\appendix 

\section{Computing the first few values of $\bet(n)$}
\label{app:computing_bet(n)}

In this section we compute the values of $\bet(n)$ for all $3 \leq n \leq 7$ (Figure~\ref{fig:tableOfValues}). Lemma~\ref{lem:bet(3,4)} deals with $n=3$ and $n=4$, and Lemma~\ref{lem:bet(5,6,7)} deals with $n=5$, $n=6$ and $n=7$.

\begin{lem}
\label{lem:bet(3,4)}
For $n=3,4$ we have $\bet(n) = n$.
\end{lem}
\begin{proof}
We consider the two cases $n=3$ and $n=4$ separately.
\begin{itemize}
\item For $n=3$, any ordering between-satisfies $2$ constraints and the total number of constraints is $6$. Hence $\bet(3) \geq 3$. To see that $\bet(3) \leq 3$, simply observe that 
	$$ \{ (1,2,3) , (2,3,1) , (3,1,2) \} \in \BET(3). $$
\item For $n=4$, any ordering between-satisfies $8$ constraints and the total number of constraints is $24$. Therefore if $\bet(4)=3$, then for any $\Phi \in \BET(4)$ with $|\Phi|=3$, we must have each constraint between-satisfied by exactly one ordering in $\Phi$. Now consider the constraints: 
	$$ (1,4,2) \; , \;\; (2,4,3) \; \mbox{ and } \; (3,4,1). $$ 
Any ordering either between-satisfies either none of them or exactly two of these three constraints. So it is not possible for $\Phi$ to between-satisfy all three of these constraints unless one of them is between-satisfied by more than one ordering in $\Phi$. This is a contradiction, so $\bet(4) \geq 4$. To show $\bet(4)=4$, simply observe that 
	$$  \{ (1,2,3,4) , (2,3,4,1) , (3,4,1,2) , (4,1,2,3) \} \in \BET(4).$$
\end{itemize}
\end{proof}

\begin{figure}
 \centering
 \begin{tabular}{|cc|ccccc|}
 \hline
 & & \multicolumn{5}{c|}{$A(i)$} \\
 & & 0 & 1 & 2 & 3 & 4 \\
 \hline
 \parbox[t]{2mm}{\multirow{5}{*}{\rotatebox[origin=c]{90}{$B(i)$}}} & 0 & \ding{51} & \ding{51} & \ding{53} & \ding{53} & \ding{53} \\
 & 1 & \ding{51} & \ding{51} & \ding{53} & \ding{53} & \ding{53} \\
 & 2 & \ding{51} & \ding{51} & \ding{51} & \ding{53} & \ding{53} \\
 & 3 & \ding{51} & \ding{51} & \ding{53} & \ding{53} & \ding{53} \\
 & 4 & \ding{51} & \ding{53} & \ding{53} & \ding{53} & \ding{53} \\
 \hline
 \end{tabular}
\caption{Possible values of $A(i)$ and $B(i)$ as defined in the proof of Lemma~\ref{lem:bet(5,6,7)}.}
\label{fig:A(i),B(i)values}
\end{figure}

\begin{lem}
\label{lem:bet(5,6,7)}
For $n=5,6,7$ we have $\bet(n)=5$.
\end{lem}
\begin{proof}
Consider the order-system $\Psi = \{ \psi_1,\psi_2,\psi_3,\psi_4,\psi_5 \} \in \BET(7)$ defined by 
	\begin{align*}
	\psi_1 & = \left( 1,6,4,3,2,7,5 \right) \\
	\psi_2 & = \left( 2,6,5,4,3,7,1 \right) \\
	\psi_3 & = \left( 3,6,1,5,4,7,2 \right) \\
	\psi_4 & = \left( 4,6,2,1,5,7,3 \right) \\
	\psi_5 & = \left( 5,6,3,2,1,7,4 \right). 
	\end{align*}
Therefore $\bet(7) \leq 5$. Since $\bet$ is a non-decreasing function of $n$, it is now sufficient to show $\bet(5)>4$. To do this, we will assume there exists an order-system $\Phi = \{ \phi_1,\phi_2,\phi_3,\phi_4 \} \in \BET(5)$ and find a contradiction. Now for $i = 1,2,3,4,5$ let $A(i)$ and $B(i)$ denote the following cardinalities. 
	$$ A(i) = \left| \left\{ j \: | \: \phi_j(i) = 1 \mbox{ or } \phi_j(i)=5 \right\} \right| \quad \mbox{and} \quad 
	B(i) = \left| \left\{ j \: | \: \phi_j(i) = 3 \right\} \right| $$
Each of the orderings $\phi_1$, $\phi_2$, $\phi_3$ and $\phi_4$ contribute a $+1$ to two different $A(i)$s and a $+1$ to a single $B(i)$. So
	$$ \sum_{i=1}^5 A(i) = 8 \qquad \mbox{ and } \qquad \sum_{i=1}^5 B(i) = 4. $$
For each $i$, the possible values of $A(i)$ and $B(i)$ are given in Figure~\ref{fig:A(i),B(i)values}. These possibilities are tedious but not difficult to verify - for example it is not possible to have $A(i)=4$ because that would imply that none of the constraints $(x,i,y)$ are between-satisfied. Anyway we can see from the values given in Figure~\ref{fig:A(i),B(i)values} that $B(i) \geq 2(A(i)-1)$ for all $i$. However, this is a contradiction since 
	$$ 4 = \sum_{i=1}^5 B(i) \geq \sum_{i=1}^5 2(A(i)-1) = 6. $$
\end{proof}

\end{document}